\documentclass{article}
\usepackage{mathptmx}      %
\usepackage{amssymb}
\usepackage{amscd}
\usepackage{amsmath}
\usepackage{amsthm}
\usepackage[msc-links]{amsrefs}
\usepackage{color}
\usepackage{amsrefs}
\usepackage{enumitem}
\usepackage{scrtime}
\newcommand{\supp}{{\rm supp \;}}
\newcommand{\R}{\mathbb{R}}
\newcommand{\N}{\mathbb{N}}
\newcommand{\Z}{\mathbb{Z}}
\newcommand{\T}{\mathbb{T}}

\newcommand{\mul}{{\bf m}}

\newcommand{\norm}[1]{\left\Vert#1\right\Vert}
\newcommand{\brkt}[1]{\left(#1\right)}

\newcommand{\abs}[1]{\left|#1\right|}
\newcommand{\set}[1]{\left\{#1\right\}}
\newcommand{\dd}{\mathrm{d}}
\renewcommand{\d}{\partial}
    \newtheorem{thm}{Theorem}[section]
    \newtheorem{cor}[thm]{Corollary}
    \newtheorem{lem}[thm]{Lemma}
    \newtheorem{prop}[thm]{Proposition}
    \newtheorem{defn}[thm]{Definition}
    \newtheorem{rem}[thm]{Observation}

    \numberwithin{equation}{section}

\date{}
\title{Restriction results for multilinear multipliers on weighted settings}
\author{Salvador Rodr\'iguez-L\'opez
\footnote{
\newline
The author has been partially supported by the Grant MTM2010-14946.}\\
\small Department of Mathematics, Uppsala University\\
\small Sweden  (\texttt{salvador@math.uu.se})}

\begin{document}
\maketitle

\begin{abstract}
We obtain restriction results of K. de Leeuw's type for maximal operators defined through multilinear Fourier multipliers of either strong or weak type acting on weighted Lebesgue spaces. We give some application of our development. In particular we obtain periodic weighted results for Coifman-Meyer, H\"ormander and H\"ormander-Mihlin type multilinear multipliers.
\end{abstract}

\section{Introduction}
The study of multilinear Fourier multipliers has it origins in the  work of R. Coifman and Y. Meyer (see for instance \cite{MR0380244}) and it has been a prolific and active area of research since the innovative work of M. Lacey and C. Thiele (see e.g. \cite{MR1689336}) on the boundedness of the bilinear Hilbert transform. The literature in this subject is currently vast, so we will confine the references to those works in direct connection with the contents of this paper.

The main body of activity in multilinear theory have consisted of proving multilinear counterparts of classical linear results. Such is the case of the theory of multilinear Calder\'on-Zygmund operators (see the seminal paper \cite{MR1880324}), and of H\"ormander-Mihlin multilinear multipliers (see \cite{MR2671120,MR2948874}). More recently, a weighted theory for such operators is being developed (see \cite{MR1947875, MR2958938,MRS,MR2483720, MR2030573} and the references therein). 

Within the development of the multilinear theory, and of direct relevance to this paper, there has been quite a few studies in establishing multilinear versions of de Leeuw's type restriction results \cite{MR0174937} on Lebesgue and Lorentz spaces \cites{MR1808390, MR2471164, MR2169476, MR2037006}. More specifically D. Fan and S. Sato \cite{MR1808390}*{Theorem 3} developed a multilinear counterpart of C. Kenig and P. Tomas  \cite{MR583403} generalisation of  de Leeuw's result for maximal operators associated to a family of multipliers given by the dilations of a given one.  To be more precise, they prove, in the particular bilinear case, the  following: 

\begin{thm}\label{eq:FanSato} Let $\mul$ be for a given continuous function in $\R^{2d}$ and let $T_r$ denote the bilinear Fourier multiplier operator associated to $\mul(r\xi,r\eta)$. Suppose that $\sup_{r>0}\abs{T_r (f_1,f_2)(x)}$ is a bounded operator mapping $L^{p_1}(\R^d)\times L^{p_2}(\R^d)$ to $L^{p}(\R^d)$, with $\frac{1}{p_1}+\frac{1}{p_2}=\frac{1}{p}$, with $1\leq p_1,p_2<\infty$. Then, the same holds for the maximal operator on $L^{p_1}(\T^d)\times L^{p_2}(\T^d)$ associated to the multipliers given by $\{\mul(rk_1,rk_2)\}_{k_1,k_2\in \Z^d}$. 
\end{thm}

The authors developed also similar results for operators of weak type. Moreover, L. Grafakos and P. Hond\'ik \cite{MR2212316} obtained a generalisation of Fan and Sato's results for general families of multipliers. 

In the linear setting, weighted linear extensions of de Leeuw's results have been developed  by E. Berkson and T.A. Gillespie \cite{MR2001939}, K. Andersen and P. Mohanty \cite{Andersen_Mohanty}  and by M. Carro and the author \cites{MR2888205, MR2406690} for certain type of weights.  

The purpose of this paper is to extend the transference results of  Fan and Sato \cite{MR1808390} and Grafakos Hondz\'ik \cite{MR2212316} to multilinear maximal operators associated to a family of multipliers of either strong or weak type, acting on products of weighted Lebesgue spaces. In particular, we generalise Theorem \ref{eq:FanSato} to weighted settings for a certain family of weights. 

Note that many interesting cases of multilinear operators can map Banach Lebesgue spaces into $L^p$ spaces with $0 < p < 1$.  This is an obstruction whenever one tries to study certain properties of multilinear operators, as it prevents to use arguments where the Banach structure is crucial. The methods developed in this paper allow to get around the difficulties that can arise from lack of convexity in the target spaces.

The paper is organised as follows: In \S \ref{section:notation} we introduce the basic notation and state our main result (Theorem \ref{theorem:A_p_restrict_T_general} below), whose proof we develop in \S \ref{sect:Proof}. Finally, in the last section, to illustrate the applications of our main result,  we give periodic counterparts of the results of  L. Grafakos and R. Torres \cite{MR1947875} on Coifman-Meyer multipliers,  M. Fujita and N. Tomita \cite{MR2958938} regarding H\"ormander-Mihlin type multilinear multipliers and N. Michalowski, D. Rule and W. Staubach \cite{MRS} regarding multipliers in the bilinear H\"ormander class $S^m_{\rho,0}$. 

\section{Notations and main result}\label{section:notation}

 We shall denote by $\T^d$ the topological group $\R^d/\Z^d$, which can be identified with the cube $[0,1)^d$ or the cube $[-\frac{1}{2},\frac{1}{2})^d$ eventually.  Functions on
$\T^d$ will be identified with functions on $\R^d$ which are
$1$-periodic in each variable. 

When $0<p<\infty$ the Lebesgue spaces $L^p(\R^d)$ and $L^p(\T^d)$ will be the usual spaces corresponding, respectively, to Lebesgue measure on $\R^d$ and to $\T^d$.

A weight on $\R^d$ is a locally integrable function with respect to the Lebesgue measure $w:\R^d \to [0,\infty)$ such that $0<w<\infty$ \, a.e. We shall write $L^p(\R^d,w)$ for the space of functions $f$ defined by the quasi-norm
\[
	\norm{f}_{L^p(\R^d,w)}=\brkt{\int \abs{f(x)}^p w(x)\dd x}^{\frac{1}{p}}.
\]
By abuse of notation, for any measurable set $E$, we will write $w(E)=\int_E w(x)\dd x$. 
We consider the weak-Lebesgue space $L^{p,\infty}(\R^d,w)$ to be the space of functions defined by the  quasi-norm 
\[
	\norm{f}_{L^{p,\infty}}=\sup_{t>0} t w\brkt{\set{x:\; |f(x)|>t}}^{1/p}.
\]
Whenever $w$ is $1$-periodic on each variable, similar definitions hold for the spaces $L^p(\T^d,w)$ and $L^{p,\infty}(\T^d,w)$.

We shall designate by $\mathcal{C}_c(\R^d)$, $\mathcal{C}^\infty (\R^d)$ and $\mathcal{C}^\infty_c(\R^d)$ the spaces of continuous functions with compact support, the space of infinitely differentiable functions, and the space of  infinitely differentiable functions with compact support respectively.  A function $g:\T^d\to \mathbb{C}$
such that for a finitely supported sequence $\{a_k\}_{k\in \Z^d}$ of
complex numbers is written as
\[
    g(x)=\sum_{k\in \Z^d} a_k e^{2\pi i kx},
\]
is called a trigonometric polynomial and we write $g\in P(\T^d)$. 

As is well known $\mathcal{C}_c^\infty(\R^d)$ is a dense subset in $L^p(\R^d,w)$ for  any weight $w$ and any $1\leq p<\infty$, and also that $P(\T^d)$ is  dense in $L^p(\T^d, w)$ for any weight $w$ $1$-periodic on each coordinate. Observe that for such weight the local integrability implies that $w\in L^1(\T^d)$.

For any function $f$,  we shall denote by $\widehat{f}$ the Fourier transform of $f$, whenever it is well defined.  By abuse of notation, we will represent also by $\widehat{g}$  the Fourier transform for any periodic function $g$. Thus, for $f\in L^1(\R^d)$ 
\[
	\widehat{f}(\xi)=\int_{\R^d} f(x)e^{2\pi i \xi x}\dd x, \quad {\text{for all}}\; \xi\in \R^d
\]
and for $g\in L^1(\T^d)$
\[
	\widehat{g}(k)=\int_{[0,1)^d} g(x)e^{2\pi i k x}\dd x,  \quad {\text{for all}}\; k\in \Z^d.
\]
From now on, $N$ stands for a natural number greater or equal to $2$. 
For a given function $\mul\in L^\infty(\R^{Nd})$, we denote by $T_\mul$ the $N$-linear operator given by
\[
	T_{\mul}(f_1,\ldots,f_N)(x)=\int_{\R^{Nd}} \mul(\xi_1,\ldots,\xi_N) \prod_{l=1}^N \widehat{f_l}(\xi_l) e^{2\pi i (\xi_1+\ldots+\xi_N)x}\, \dd \vec{\xi},\quad x\in \R^d,
\]
for $f_1,\ldots f_N\in \mathcal{C}^\infty_c(\R^d)$. We say that $T_\mul$ is a ($N$-linear) multiplier.  Similarly, for a given $\mul\in L^\infty(\Z^{Nd})$, we denote by $\mathfrak{T}_{\mul}$ the $N$-linear operator defined  by 
\[
	 \mathfrak{T}_{\mul}(g_1,\ldots, g_N)(x)=\sum_{k_1,\ldots, k_N\in \Z^{d}} \mul(k_1,\ldots, k_N) \prod_{l=1}^N \widehat{g_l}(k_l) e^{2\pi i (k_1+\ldots+k_N)x},\quad x\in \T^{d},
\]
for $g_1,\ldots, g_N\in P(\T^d)$. 

Let $\mathcal{F}$ denote a countable fixed set of indices. For a given family $\{\mul_j\}_{j\in \mathcal{F}}$ of bounded functions in $\R^{Nd}$, we consider the maximal operators associated given by
\begin{equation}\label{eq:maximal}
	M (f_1,\ldots, f_N)(x)=\sup_{j\in \mathcal{F}} \abs{T_{\mul_j}(f_1,\ldots, f_N)(x)}.
\end{equation}
Observe that if there exists a function $K\in L^1(\R^{2d})$ such that $\mul=\widehat{K}$, then $T_{\mul}$ coincides with the operator given by
\begin{equation}\label{eq:BK}
	B_{K}(f_1,\ldots, f_N)(x)=\int_{\R^{Nd}}  K(x_1,\ldots, x_N) \prod_{j=1}^N f_j(x-x_j) \, \dd x_1\ldots \dd x_N.
\end{equation}
Similarly, we define for $\{\mul_j\}_{j\in \mathcal{F}}\subset L^{\infty}(\Z^{Nd})$ the associated maximal operator as
\begin{equation}\label{eq:maximal_periodic}
	\mathfrak{M}(g_1,\ldots, g_N)(x)={ \sup_{j\in \mathcal{F}} \abs{ \mathfrak{T}_{\mul_j}(g_1,\ldots, g_N)}}.
\end{equation}
For simplicity on the notation, we will omit the dependency  on $\mathcal{F}$ and $\{\mul_j\}_{j\in \mathcal{F}}$ of the definition of $M$ and $\mathfrak{M}$. From now on, unless stated to the contrary, we will restrict our attention to indices $1\leq p_1,\ldots p_N<\infty$ and $0<p<\infty$ satisfying
\begin{equation}\label{eq:relacion}
    \frac{1}{p_1}+\ldots + \frac{1}{p_N}=\frac{1}{p}.
\end{equation}

\begin{defn}\label{def:normalized}We say that $\mul\in L^\infty(\R^{dN})$ is {\it normalized} if for any $\xi_1,\ldots,\xi_N\in\R^d$, 
\[
    \lim_n \mul*\Phi_n(\xi_1,\ldots, \xi_N)=\mul(\xi_1,\ldots, \xi_N),
\]
where $\varphi_n(x)=\varphi(x/n)$, $\varphi\in \mathcal{C}^\infty_c(\R^d)$,
    $\widehat{\varphi}\geq 0$ and $\norm{\widehat{\varphi}}_1=1$, $	\Phi_n(\xi_1,\ldots, \xi_N)=\prod_{l=1}^N \varphi_n(\xi_l)$ and $*$ denotes the usual convolution in $\R^{Nd}$.
\end{defn}

\begin{rem}
It is easy to see that any continuous and bounded function is also normalized. Observe that in particular, for any normalized function $\mul$, the point-wise evaluation 
\[
	\mul\vert_{\Z^{Nd}}=\set{\mul(k_1,\ldots,k_N)}_{k_1,\ldots,k_N\in \Z^d},
\]
makes sense as the point-wise limit of continuous functions. 
\end{rem}

The main results of this paper  concerns transference of the boundedness of maximal normalized multipliers acting on weighted Lebesgue space and can be stated as follow.

\begin{thm}\label{theorem:A_p_restrict_T_general} Let $w, w_l$ for $l=1,\ldots, N$ be $1$-periodic weights and let   $\set{\mul_j}_{j\in \mathcal{F}}$ be a family of normalized functions.  Let $M$ be the associated maximal operator defined as in \eqref{eq:maximal} and let $\mathfrak{M}$ be the maximal operator as in \eqref{eq:maximal_periodic} associated to $\set{\mul_j\vert_{\Z^{Nd}}}_{j\in \mathcal{F}}$. 
\begin{enumerate}[leftmargin=1.7em]
	\item If there exists a constant $\mathfrak{N}$ such that 
	\begin{equation}
	\label{eq:strong}
	\norm{M(f_1,\ldots, f_N)}_{L^p\brkt{\R^d,w}} \leq \mathfrak{N} \prod_{l=1}^N \norm{f_l}_{L^{p_l}\brkt{\R^d,w_l}},
	\end{equation}
	for any $f_l\in L^{p_l}(\R^d,w_l)$, $l=1,\ldots, N$, then
	\begin{equation}
	\label{eq:strong_periodic}
	\norm{\mathfrak{M}(g_1,\ldots, g_N)}_{L^{p}\brkt{\T^d,w}} \leq \mathfrak{c}_{\vec p}\mathfrak{N} \prod_{l=1}^N \norm{g_l}_{L^{p_l}\brkt{\T^d,w_l}}
	\end{equation}
	  for any $g_l\in L^{p_l}(\T^d,w_l)$, $l=1,\ldots, N$.
		\item If there exists a constant $\mathfrak{N}$ such that 
	\begin{equation}
	\label{eq:weak}
	\norm{M(f_1,\ldots, f_N)}_{L^{p,\infty}\brkt{\R^d,w}} \leq \mathfrak{N} \prod_{l=1}^N \norm{f_l}_{L^{p_l}\brkt{\R^d,w_l}},
	\end{equation}
	for any $f_l\in L^{p_l}(\R^d,w_l)$, $l=1,\ldots, N$, then
	\begin{equation}
	\label{eq:weak_periodic}
	\norm{\mathfrak{M}(g_1,\ldots, g_N)}_{L^{p,\infty}\brkt{\T^d,w}} \leq \mathfrak{c}_{\vec p}\mathfrak{N} \prod_{l=1}^N \norm{g_l}_{L^{p_l}\brkt{\T^d,w_l}},
	\end{equation}
	for any $g_l\in L^{p_l}(\T^d,w_l)$, $l=1,\ldots, N$.
\end{enumerate}
In both cases, $c_{\vec{p}}$ is a constant depending only on $\vec p=(p,p_1,\ldots,p_N)$.
\end{thm}

\begin{defn}\label{def:best} For a given family of normalized functions  $\set{\mul_j}_{j\in \mathcal{F}}$, we shall denote by 
$
	\mathfrak{N}(\set{\mul_j}_{j\in \mathcal{F}}), (\text{respect. } \mathfrak{N}^w(\set{\mul_j}_{j\in \mathcal{F}})) 
$
the least constant satisfying \eqref{eq:strong} (resp. \eqref{eq:weak}).
\end{defn}

\begin{rem} Observe that  the previous result can be applied also to the case of a single multiplier by taking $\mathcal{F}$ to consist in one element. Observe also that for $w=w_l=1$ the previous result recovers Fan and Sato's \cite{MR1808390}*{Theorem 3} and Grafakos and Honz{\'{\i}}k \cite{MR2212316}*{Thm. 2.2}. 
\end{rem}

\section{Proof of Theorem \ref{theorem:A_p_restrict_T_general}}\label{sect:Proof}

For the sake of simplicity, in the exposition, we shall restrict our proofs to the bilinear case ($N=2$) as it contains the main ideas of the development and the arguments  can be easily extended to any $N\geq 2$.

We start by proving a weaker version of Theorem \ref{theorem:A_p_restrict_T_general}, where we assume stronger conditions on the multipliers. To this end, we need to recall the so called Kolmogorov condition (see \cite{MR807149}*{p. 485}). Let $\mathcal{M}$ be $\R^d$ or $\T^d$. For any $q<p$, we have the inequalities
\begin{equation}\label{eq:kolmogorov}
   \norm{f}_{L^{p,\infty}(\mathcal{M},w)}\leq \sup \norm{f\chi_E}_{L^q(\mathcal{M},w)}w(E)^{1/p-1/q}\leq c_{p,q}\norm{f}_{L^{p,\infty}(\mathcal{M},w)},
\end{equation}
where the supremum is taken on the family of sets $E$  with $0<w(E)<\infty$ and $c_{p,q}^q={{p}/({p-q})}$.

\begin{thm}\label{theorem:A_p_restrict_T} Let $w, w_l$ for $l=1,2$ be $1$-periodic weights and let   $\set{\mul_j}_{j\in \mathcal{F}}$ satisfying  that,  for each $j$, there exists $K_j\in L^{1}(\R^{Nd})$ with compact support such that $\widehat K_j(\xi)=\mul_j(\xi)$ for every $\xi\in\R^{Nd}$. Let $M$ be the associated maximal operator defined as in \eqref{eq:maximal} and let $\mathfrak{M}$ be the maximal operator as in \eqref{eq:maximal_periodic} associated to $\set{\mul_j\vert_{\Z^{2d}}}_{j\in \mathcal{F}}$. 

Assume that there exists a constant $\mathfrak{N}$ such that 
 \eqref{eq:strong} (respectively \eqref{eq:weak}) holds. Then \eqref{eq:strong_periodic} (resp. \eqref{eq:weak_periodic}) holds
where $\mathfrak{c}_{\vec{p}}=1$ (resp. $\mathfrak{c}_{\vec{p}}=\inf_{q<p} c_{p,q}$).
\begin{proof} 

By Fatou's Lemma, without loss of generality we can assume that $\mathcal{F}$ is a finite family of indices $\mathcal{F}=\set{1,\ldots, J}$, where $J\in \N$.  By sake of brevity  we are going to prove only the weak case. The strong case is obtained in a similar way with minor modifications in the proof, so we omit the details. 

It is easy to see that \eqref{eq:weak} yields that for every $\theta\in [0,1)^d$ 
    \begin{equation}\label{eq:technic_7}
        \norm{\sup_{1\leq j\leq J}\abs{B_{K_j}(f_1,f_2)}}_{L^{p,\infty}(\R^d,w(\cdot+\theta))}\leq \mathfrak{N}
        \prod_{l=1,2}\norm{f_l}_{L^p(\R^d, w_l(\cdot+\theta))}.
    \end{equation}
where $B_{K_j}(f_1,f_2)(x)=\int_{\R^{2d}}  K_j(x_1,x_2)f_1(x-x_1) f_2(x-x_2)\, \dd x_1\dd x_2$ as in \eqref{eq:BK}.    Let $g_l(\theta)=\sum_k a_k^l e^{2\pi i k\theta}\in P(\T^d)$  for $l=1,2$.  Consider
    \[
        \begin{split}
       T_{K_j}(g_1,g_2)(\theta) &=\int_{\R^{2d}} K_j(x_1,x_2)  \prod_{l=1,2} g_l(\theta-x_l)\, \dd x_1\dd x_2. %
        \end{split}
    \]
Observe that $T_{K_j}$ coincides with the  bilinear multiplier operator $\mathfrak{T}_{\mul_j\vert_{\Z^{2d}}}$, where $\mul_j\vert_{\Z^{2d}}$ is the sequence given by $\{\mul_j(k_1,k_2)\}_{k_1,k_2\in \Z^{d}}$.

    Let  $r>0$ big enough such that $\supp K_j\subset Q_r\times Q_r$ for $j=1,\ldots, J$ where $Q_r=(-r,r)^d$.   Fix any $q<p$ and for any measurable   $E\subset [0,1)^d$,  define $\tilde E=\cup_{k\in \Z^d}E+k$ as its $1$-periodic extension and, fixed $\theta\in \T^d$ let $E_\theta=\set{x\in \R^d:\; x+\theta\in \tilde E}$. Denote by $ R_{x} g(\theta)=g(\theta+x)$. The translation invariance of the Lebesgue measure yields
    \begin{equation*}
    \begin{split}
     &\norm{\sup_{1\leq j\leq J} \abs{T_{K_j}(g_1,g_2)}\chi_E }_{L^q(\T^d, w)}^q= \\
     &\qquad =
        \int_{\T^d} \sup_{1\leq j\leq J} \abs{R_xT_{K_j} (g_1,g_2)(\theta)}^q w(x+\theta)\chi_{\tilde E}(x+\theta)\; \dd \theta.
        \end{split}
    \end{equation*}
    for every $x\in \R^d$.  Therefore, for every  $s>0$, integration yields
 \begin{equation}\label{eq:technic_I}
\begin{split}
        &\left\Vert{\sup_{1\leq j\leq J} \abs{T_{K_j}(g_1,g_2)}\chi_E}\right\Vert_{L^q(\T^d,w)}^q
        \\
        &\quad =\frac{1}{(2s)^d} \int_{\T^d} \int_{Q_s\cap E_\theta}\sup_{1\leq j\leq N}\abs{R_x  {T}_{K_j} (g_1,g_2)(\theta)}^q
        w(x+\theta)\, \dd x \dd \theta.
\end{split}
\end{equation}
Since $\supp K_j\subset Q_r\times Q_r$ for $j=1,\ldots,  J$,  it follows that we can write
$$
R_xT_{K_j} (g_1,g_2)(\theta)=B_{K_j}\brkt{R_{(\cdot)} g_1(\theta) \chi_{Q_{r+s}}, R_{(\cdot)} g_2(\theta) \chi_{Q_{r+s}}}(x),
$$
for any $x\in Q_s$. 
Therefore, by (\ref{eq:technic_7}) and  (\ref{eq:kolmogorov}), the term in \eqref{eq:technic_I} is bounded by
\[
    \frac{(c_{p,q} \mathfrak{N})^q }{(2s)^d} \int_{\T^d} \set{\int_{E_\theta\cap Q_s}  w(x+\theta)\, \dd x}^{1-\frac{q}{p}} \prod_ {l=1,2}\left\{\int_{Q_{r+s}} \abs{R_{x}g_l(\theta)}^{p_l} w_l(x+\theta)\; \dd x\right\}^{\frac q {p_l}}\dd \theta. %
\]
Since $1=\brkt{1-\frac{q}{p}}+\frac{1}{p_1}+\frac{1}{p_2}$, H\"older's inequality yields that the previous term is bounded by 
\[
	\frac{c_{p,q}^q \mathfrak{N}^q}{(2s)^d} \set{\int_{\T^d} \int_{Q_s\cap E_\theta} w(x+\theta)
        \dd x\,\dd \theta}^{1-\frac q p} \prod_{l=1,2}\set{\int_{\T^d} \int_{Q_{r+s}}
        \left|R_xg_l(\theta)\right|^{p_l} w_l(x+\theta) dt \dd \theta}^{\frac q {p_l}}
\]
Exchanging the order of integration, the term in the first curly bracket is equal to 
\[
	\set{\int_{\T^d} \int_{Q_s}\chi_{\tilde{E}}(x+\theta) w(x+\theta)
        \dd x\,\dd \theta}^{1-\frac q p}=w(E)^{1-\frac q p} (2s)^{d(1-\frac q p)},
\]
and we have 
\[
	\set{\int_{\T^d} \int_{Q_{r+s}}
        \left|R_xg_l(\theta)\right|^{p_l} w_l(x+\theta) dt \dd \theta}^{\frac q {p_k}}=(2(r+s))^{\frac q {p_l}} \norm{g_l}_{L^{p_l}(\T^d,w)}^q.
\]
Thus, for any $s>0$,
\[
       \left\Vert{\sup_{1\leq j\leq J} \abs{T_{K_j}(g_1,g_2)}\chi_E}\right\Vert_{L^q(\T^d,w)}\leq {c_{p,q}\mathfrak{N} \brkt{\frac{r+s}{s}}^{\frac{d} p}} w(E)^{{\frac{1}{q}-\frac {1} p}}\prod_{l=1,2} \norm{g_l}_{L^p(\T^d,w_l)}.
\]
Therefore, taking $s\to +\infty$ and using (\ref{eq:kolmogorov}) we have
\[
    \norm{\sup_{1\leq j\leq J} \abs{T_{K_j}(g_1,g_2)}}_{L^{p,\infty}(\T^d,w)}\leq c_{p,q} \mathfrak{N}  \norm{g_1}_{L^{p_1}(\T^d,w_1)} \norm{g_2}_{L^{p_2}(\T^d,w_2)},
\]
from where the result follows considering $\inf_{q<p} c_{p,q}$.
\end{proof}
\end{thm}

The next step is to weaken the hypothesis assumed on the multipliers $\mul_j$. To this end we shall give some previous technical lemmas. The following result holds for general measure spaces  $(\mathcal{M},\mu)$ and $(\mathcal{M}_j,\mu_j)$ $j=1,2$. 
\begin{thm}\label{eq:Marcinkiewicz_Zygmund} Let $\{T_j\}_{j}$ be a countable family of bilinear operators which satisfies that there exists a constant $\mathfrak{N}$ such that for any $f_l\in L^{p_l}(\mathcal{M}_l,\mu_l)$ with  $l=1,2$
\begin{equation}\label{MZboundedness}
    \norm{\sup_{j}\abs{T_j ( f_1,f_2)}}_{L^{p}(\mathcal{M},\dd \mu)}\leq\mathfrak{N} \norm{f_1}_{L^{p_1}(\mathcal{M}_1,\dd \mu_1)} \norm{f_2}_{L^{p_2}(\mathcal{M}_2,\dd \mu_2)},
\end{equation}
where $p_1,p_2\geq p$. Then 
\begin{equation}\label{MZ}
    \norm{\sup_{j}\left(\sum_{k,l} |T_j (f_{1,k},f_{2,l})|^2\right)^{1/2}}_{L^{p}(\mathcal{M},\dd \mu)} \hspace{-1cm}\leq\mathfrak{c}_{p,p_1,p_2}\mathfrak{N} \prod_{l=1,2} \norm{\left(\sum_k
    |f_{l,k}|^2\right)^{1/2}}_{L^{p_l}(\mathcal{M}_l,\dd \mu_l)}\hspace{-1cm},
\end{equation}
where $\mathfrak{c}_{p,p_1,p_2}$ is a constant depending on $p,p_1,p_2$. 
\end{thm}
\begin{proof}
	Without loss of generality we can reduce us to prove the result for $j$ in a finite set of indices $\{1,\ldots, J\}$, and for $\{f_{l,k}\}$ with a finite number of elements for $l=1,2$.  Khintchine's bilinear inequality \cite{MR0290095}*{Appendix
    D}, bilinearity and \eqref{MZboundedness} yield that the left hand side term in \eqref{MZ} is bounded by 
   \begin{eqnarray*}
            & & \frac{1}{A_p^2} \norm{\brkt{\sup_{j}\iint_{[0,1]^2} \abs{\sum_{k,l}r_k(s)r_l(t)
            T_j(f_{1,k},f_{2,l})}^p \, {\dd s }\, {\dd t }}^{1/p}}_{L^p}\\
            &\leq& \frac{1}{A_{p}^2}\brkt{\iint_{[0,1]^2} \norm{\sup_j\abs{T_j\brkt{\sum_k r_k(s) f_{1,k}, \sum_l r_l(t)f_{2,l} }}}_{L^p}^p
             \, {\dd s }\, {\dd t }}^{1/p}\\
            &\leq& \frac{\mathfrak{N}}{A_{p}^2}\brkt{\int_0^1 \norm{\sum_j r_j(s)
            f_{1,k}}_{L^{p_1}}^p\, {\dd s }}^{1/p}\brkt{\int_0^1 \norm{\sum_k r_k(t)
            f_{2,k}}_{L^{p_2}}^p\, {\dd t }}^{1/p},
    \end{eqnarray*}
    for a certain universal constant $A_p$ depending only on $p$. 
    Since for $l=1,2$, $p_l\geq p$, H\"older inequality and Khintchine's inequality yield 
    \[
    	\begin{split}
        \brkt{\int_0^1 \norm{\sum_k r_k(s)
            f_{l,k}}_{L^{p_l}}^p\!\!\!\!  {\dd s }}^{\frac 1 p}&\leq \brkt{\int_0^1 \int_{\R^d} \abs{\sum_k r_k(s)
            f_{l,k}(x)}^{p_l} \dd\mu_l(x){\dd s }}^{\frac 1 {p_l}}\\
            &\leq B_{p_l} \norm{\brkt{\sum_{k}
            \abs{f_{l,k}}^2}^{1/2}}_{L^{p_l}},
	\end{split}
    \]
    for a certain constant $B_{p_l}$ depending only on $p_l$. Hence the result follows with $\mathfrak{c}_{p,p_1,p_2}=B_{p_1}B_{p_2}/{A_p}$.
\end{proof}

A direct application of the previous theorem in combination with \eqref{eq:kolmogorov} yields the following result.
\begin{cor} Let $\{T_j\}_{j}$ be a countable family of bilinear operators which satisfies that there exists a constant $\mathfrak{N}$ such that for any $f_l\in L^{p_l}$ with  $l=1,2$
\[
    \norm{\sup_{j}\abs{T_j ( f_1,f_2)}}_{L^{p,\infty}(\R^d,\dd \mu)}\leq\mathfrak{N} \norm{f_1}_{L^{p_1}(\R^d,\dd \mu_1)} \norm{f_2}_{L^{p_2}(\R^d,\dd \mu_2)},
\]
where $p_1,p_2\geq p$. Then 
\[
    \norm{\sup_{j}\left(\sum_{k,l} |T_j (f_{1,k},f_{2,l})|^2\right)^{1/2}}_{L^{p,\infty}(\R^d,\dd \mu)}\leq \mathfrak{c}_{p,p_1,p_2}\mathfrak{N} \prod_{l=1,2} \norm{\left(\sum_k
    |f_{l,k}|^2\right)^{1/2}}_{L^{p_l}(\R^d,\dd \mu_l)},
\]
where $\mathfrak{c}_{p,p_1,p_2}$ is a constant depending on $p,p_1,p_2$. 
\end{cor}
\begin{lem}\label{lem:regularidad_peso_debil}  Let $0<p,p_1,p_2<\infty$ such that $\frac{1}{p}=\frac{1}{p_1}+\frac{1}{p_2}$. Let  $T$ be any bounded operator from $L^{p_1}(\R^d, w_1)\times L^{p_2}(\R^d, w_2)$ to $L^{p}(\R^d, w)$ (resp. $L^{p,\infty}(\R^d, w)$), with norm $\mathfrak{N}$. Suppose that $T$ satisfies that
\[
	\tau_y T(f_1,f_2)=T(\tau_y f_1,\tau_y f_2),\qquad  \text{for any $y\in \R^d$}.
\] 
 Then, for any nonnegative function $\psi\in \mathcal{C}_c(\R^d)$, $T$ is bounded from $L^{p_1}(\R^d, \psi*w_1)\times L^{p_2}(\R^d, \psi*w_2)$ to $L^{p}(\R^d, \psi*w)$  (resp. $L^{p,\infty}(\R^d,\psi*w)$) with norm bounded by $\mathfrak{N}$ (resp. $\inf_{q<p} c_{p,q} \mathfrak{N}$).
\end{lem}
\begin{proof} We'll prove  only the weak case as the argument can be easily adapted to cover the strong case. Let $E$ be any measurable set in $\R^d$ such that $0<\psi*w(E)<+\infty$. Then,
    for any $q<p$,
    \[
        \begin{split}
            &\norm{T(f_1,f_2) \chi_E}_{L^q(\R^d,\psi*w)}^q\\ 
            &=\int_{E} \abs{T(f_1,f_2)(x)}^q \psi*w(x)\, \dd x=\int \psi(y) \int_E \abs{T(f_1,f_2)(x)}^q w(x-y)\, \dd x\,dy\\
                &=\int \psi(y) \int_{E-y} \abs{T(\tau_{-y}f_1,\tau_{-y}f_2)(x)}^q w(x)\, \dd x\,dy,
        \end{split}
    \]
    with $\tau_{-y}f(z)=f(z+y)$.
    Thus, by the boundedness hypothesis, \eqref{eq:kolmogorov} and H\"older's inequality, the last term in the previous expression is bounded by 
    \[
        \begin{split}
	&\mathfrak{N}^q\int \psi(y) w(E-y)^{1-\frac{q}{p}}\prod_{l=1,2}\brkt{\int \abs{\tau_{-y} f(x)}^{p_l} w_l(x)\, \dd x}^{q/p_l}\,\dd y\\
                &\leq c_{p,q}^q \mathfrak{N}^q \brkt{\psi*w(E)}^{1-\frac{q}{p}} \prod_{l=1,2}\norm{f_l}_{L^{p_l}(\R^d,g*w_l)}^q.
         \end{split}
    \]
    Then, the result follows by \eqref{eq:kolmogorov} and by taking the infimum for $q<p$.
\end{proof}

\begin{rem} Although we are not going to use this property here, let us observe that the  previous lemma implies that if $(w,w_1,w_2)\in A_{\vec{p}}$ then, $(g*w,g*w_1,g*w_2)\in A_{\vec{p}}$ for any $g\in \mathcal{C}_c(\R^d)$ (see \cite{MR2483720} for the definiton and properties of these classes of weights).
\end{rem}
The next lemma is the maximal multilinear counterpart of \cite{MR2001939}*{Theorem 2.8}. We shall mention that it is an immediate consequence of Minkowskii's inequality, as long as the target space is normable, but for the general set of indices considered the convexity of the target space fails. 

\begin{prop}\label{prop:good_behaviour} Let $\varphi \in L^{1}(\R^d)$ and  $\set{\mul_j}_j\subset L^\infty(\R^d)$.  Then $\set{\varphi\otimes \varphi*\mul_j}_j$ satisfies 
\begin{eqnarray}\label{eq:good_behaviour}
    \mathfrak{N}\brkt{\set{(\varphi\otimes \varphi)*\mul_j}_j}\leq \mathfrak{c}_{\vec{p}}||\varphi||_{L^{1}(\R^d)}^2\mathfrak{N}\brkt{\set{\mul_j}_j},\\
     \label{eq:good_behaviour_weak}\mathfrak{N}^w\brkt{\set{(\varphi\otimes \varphi)*\mul_j}_j}\leq \mathfrak{c}_{\vec{p}}||\varphi||_{L^{1}(\R^d)}^2\mathfrak{N}^w\brkt{\set{\mul_j}_j}
\end{eqnarray}
where $\mathfrak{c}_{\vec{p}}$ is a constant depending only on $\vec{p}=(p,p_1,p_2)$. 
\end{prop}
\begin{proof} For simplicity we will prove only the weak case as it contains the main ideas of the proof. We leave the details of the strong case to the reader. 
Without loss of generality, we can assume that $\{\mul_j\}_j$ is a
finite family of bilinear multipliers of cardinal say $J\in\N$. Fixed $f_1,f_2\in
\mathcal{C}_c^\infty(\R^d)$,
\[
	\begin{split}
    \int &\brkt{(\varphi\otimes \varphi)*\mul_j}(\xi,\eta) \widehat{f_1}(\xi)\widehat{f_2}(\eta)e^{2\pi i (\xi+\eta) x}\; d\xi\\ 
    &=\int \varphi(\xi) \varphi(\eta)e^{2\pi i x (y+z)}T_{\mul_j}
    \brkt{e^{-2\pi i \xi\cdot } f_1,e^{-2\pi i \eta\cdot } f_2}(x)\; \dd \xi\dd \eta.
	\end{split}
\]
Hence,
\begin{equation}\label{eq:technic}
	\begin{split}
   &\abs{T_{(\varphi\otimes \varphi)*\mul_j} (f_1,f_2)(x)}\leq S_j^{\varphi}(f_1,f_2)(x)\\
   &:=\iint \abs{\varphi(\xi)}\abs{\varphi(\eta)} \abs{T_{\mul_j}
    \brkt{e^{-2\pi i \xi\cdot } f_1,e^{-2\pi i \eta\cdot } f_2}(x)}\; \dd \xi\dd \eta.
    \end{split}
\end{equation}
Observe that if $p>1$, since $L^{p,\infty}$ is a Banach space, Minkowski integral inequality applied to the last expression would conclude the result with $\mathfrak{c}_{\vec{p}}=1$. So, we can assume without loss of generality that $0<p\leq 1$. 

Let us first assume that $\varphi\in L^1(\R^d)$ is supported on a
compact set $\mathcal{K}$. Let 
\[
    F_{j,x}(\xi,\eta)=T_{\mul_j}(e^{-2\pi i \xi \cdot} f_1, e^{-2\pi i \eta \cdot} f_2)(x),\quad \xi,\eta \in \R^d.
\]
It is easy to see that $\xi,\zeta, \eta,\gamma\in\R^d$, $x\in \R^d$
\begin{eqnarray*}
    \lefteqn{\left|F_{j,x}(\xi,\eta) - F_{j,x}(\zeta,\gamma)\right|\leq}\\
    &\leq  \norm{\mul_j}  \brkt{\norm{\widehat{f_1}-\tau_{\xi-\zeta}\widehat{f_1}}_{L^1(\R^d)}
    \norm{\widehat{f_2}}_{L^1(\R^d)}+\norm{\widehat{f_2}-\tau_{\eta-\gamma}\widehat{f_2}}_{L^1(\R^d)}
    \norm{\widehat{f_1}}_{L^1(\R^d)}},
\end{eqnarray*}
where $\tau_\xi$ stands for the translation operator.  Then the uniform continuity of translations  in $L^1(\R^d)$ and a compactness argument yield that, for each $k\in\N\setminus\{0\} $, there exists a finite family
$\set{V_l^k}_{l=1}^{I_k}$ of pairwise disjoint covering of $\mathcal{K}$ given by measurable sets such that $\mathcal{K}\subset  \biguplus_{l=1}^{I_k} V_l^k$ and, if $l=1,\ldots,I_k$ and $\xi,\zeta \in V_l^{k}$ then
    \begin{equation}\label{eq:continuity}
        \sup_{1\leq j\leq J}\sup_{x,\eta\in \R^d}\abs{F_{j,x}(\xi,\eta)-F_{j,x}(\zeta,\eta)}+\abs{F_{j,x}(\eta,\xi)-F_{j,x}(\eta,\zeta)}\leq 1/k.
    \end{equation}

For each $k\geq 1$ let $\{V_l^k\}_{l=1}^{I_k}$ be the family of pairwise disjoint sets
given above. For each $l$, select $\xi_{l}^k\in
V_{l}^k$. Then, for every $\xi\in \mathcal{K}$ and any $k\geq 1$,
there exists a unique $l\in \set{1,\ldots,I_k}$ such that $\xi\in
V_{l}^k$ and hence \eqref{eq:continuity} yields
\[
   \sup_{1\leq j\leq J} \sup_{x,\eta\in \R^d}\abs{F_{j,x}(\xi,\eta)-F_{j,x}(\xi_l^k,\eta)}\leq \frac{1}{k}.
\]
Thus, by  (\ref{eq:technic}),
\[
	S_j^{\varphi}(f_1,f_2)(x)\leq \frac{\norm{\varphi}_{L^1(\R^d)}^2}{k}+\sum_{l=1}^{I_k}\lambda_{l}^k \int \abs{\varphi(\eta)} \abs{T_{\mul_j}
    \brkt{e^{-2\pi i \xi_{l}^k \cdot} f_1, e^{-2\pi i \eta \cdot} f_2}(x)}\dd \eta, 
\]
where  $\lambda_{l}^k=\int_{V_{l}^k} \abs{\varphi(\xi)}\, \dd \xi$. Repeating the same argument we obtain that 
\[
	S_j^{\varphi}(f_1,f_2)(x)\leq \frac{2\norm{\varphi}_{L^1(\R^d)}^2}{k}+\sum_{l,m=1}^{I_k}\lambda_{l}^k\lambda_{m}^k  \abs{T_{\mul_j}
    \brkt{e^{-2\pi i \xi_{l}^k \cdot} f_1, e^{-2\pi i \xi_{m}^k  \cdot} f_2}(x)},
\]
which yields 
\[
	\sup_{1\leq j\leq J} S_j^{\varphi}(f_1,f_2)(x)\leq \liminf_k \sup_{1\leq j\leq J} \sum_{l,m=1}^{I_k}\lambda_{l}^k\lambda_{m}^k  \abs{T_{\mul_j}
    \brkt{e^{-2\pi i \xi_{l}^k \cdot} f_1, e^{-2\pi i \xi_{m}^k  \cdot} f_2}(x)}
\]
Chauchy-Schwarz inequality yields that the term in the right hand side is bounded by
\[
    \norm{\varphi}_{L^1(\R^d)}^2 \left(\sum_{l,m=1}^{I_k} \abs{T_{\mul_j}\brkt{\sqrt{\lambda_{l}^k}\, e^{-2\pi i \xi_{l}^k
    \cdot} f_1, \sqrt{\lambda_{l}^k}\, e^{-2\pi i \xi_{m}^k
    \cdot} f_2}(x)}^2 \right)^{1/2},
\]
where we have used that $\sum_{l=1}^{I_k}
\lambda_{l}^k=\int_{\biguplus_{l=1}^k V_{l}^k } \abs{\varphi(y)}\;
dy=\norm{\varphi}_{L^1(\R^d)}$. Fatou's lemma yields that %
\begin{eqnarray*}
    \lefteqn{\norm{\sup_{1\leq j\leq J}S_j^{\varphi}(f_1,f_2)}_{L^{p,\infty}(\R^d,w)}\leq \norm{\varphi}_{L^1(\R^d)}\times}\\
    &\times
    \liminf_{k} \norm{\sup_{1\leq j\leq J}  \left(\sum_{l,m=1}^{I_k} \abs{T_{\mul_j}\brkt{\sqrt{\lambda_{l}^k}\, e^{-2\pi i \xi_{l}^k
    \cdot} f_1, \sqrt{\lambda_{l}^k}\, e^{-2\pi i \xi_{m}^k
    \cdot} f_2}(x)}^2 \right)^{1/2}}_{L^{p,\infty}(\R^d,w)}.
\end{eqnarray*}
Theorem \ref{eq:Marcinkiewicz_Zygmund} yields that the last term in the right hand side is bounded by the factor $\mathfrak{c}_{p,p_1,p_2} \mathfrak{N}^w\brkt{\set{\mul_j}_j}$ which multiplies
\begin{eqnarray*}
    & &\norm{\left(\sum_{l=1}^{I_k} \abs{\sqrt{\lambda_{l}^k}\, e^{-2\pi i \xi_{l}^k
    \cdot} f_1}^2\right)^{1/2}}_{L^{p_1}(\R^d, w)}\norm{\left(\sum_{l=1}^{I_k} \abs{\sqrt{\lambda_{l}^k}\, e^{-2\pi i \xi_{l}^k
    \cdot} f_2}^2\right)^{1/2}}_{L^{p_2}(\R^d, w)}\\
    & &=\norm{\varphi}_{L^1(\R^d)} \norm{f_1}_{L^{p_1}(w_1)}\norm{f_2}_{L^{p_2}(w_2)}.
\end{eqnarray*}
Using \eqref{eq:technic} and monotonicity, this implies that
\begin{equation}\label{eq:technic_2}
\begin{split}
       &\norm{\sup_{1\leq j\leq J} \abs{T_{(\varphi\otimes \varphi)*\mul_j} (f_1,f_2)}}_{L^{p,\infty}(\R^d,w)}\\
       &\leq
   \mathfrak{c}_{p,p_1,p_2}  \mathfrak{N}^w\brkt{\set{\mul_j}_j} \norm{\varphi}_{L^1(\R^d)}^2\norm{f_1}_{L^{p_1}(\R^d,w_1)}
   \norm{f_2}_{L^{p_2}(\R^d,w_2)},
  \end{split}
\end{equation}
which implies \eqref{eq:good_behaviour_weak}.

For the general case, if we consider $\varphi_n=\varphi\chi_{B(0,n)}$, we have that  $\{\sup_{1\leq j\leq J}S_j^{\varphi_n}(f_1,f_2)\}_n$ is an increasing sequence of functions which pointwise converges to $\sup_{1\leq j\leq J}S_j^{\varphi}(f_1,f_2)$. Then, the monotone convergence, \eqref{eq:technic} and the previous argument yields the result. 
\end{proof}

We will need also the two following technical lemma which proof can be found in \cite{MR2888205}.

\begin{lem}\label{lem:acotacion_peso} Let $w$ be $1$-periodic. If $\psi\in\mathcal{C}_c(\R^d)$ is nonnegative,   $\int_{\R^d} g=1$ and $\supp \psi\subset [-1/2,1/2]^d$, then $\inf_{x\in \R^d} \psi*w(x)>0$.
\end{lem}
\begin{lem}\label{lem:approximaci�n}
 Let $w\in \mathcal{C}(\T^d)$ such that $\inf_{x\in \T^d} w(x)>0$. Consider $h\in\mathcal{C}_c^\infty (\R^d)$ satisfying $0\leq h\leq 1$
 and $\int_{\R^d} h=1$ and define $h_n(x)=n^d h(n x)$. Then,
\begin{enumerate}
    \item There exists $n_0=n_0(w)\in \N$ such that
        $
            \sup_{n\geq n_0}{||\widehat{h_n}||_{M_{p,w}(\R^d)}}\leq 2^{1/p}
        $ , for any $1\leq p<\infty$, 
        where $||\widehat{h_n}||_{M_{p,w}(\R^d)}$ stands for the norm of the convolution operator given by $h_n*f$ on $L^{p}(\R^d,w)$. 
    \item $\sup_{n}||\widehat{h_n}||_{L^\infty(\R^d)}\leq 1$.
    \item For every $\xi\in \R^d$, $\lim_n\widehat{h_n}(\xi)=1$.
\end{enumerate}

\end{lem}

So, at this stage we have all the ingredients for proving our main theorem. 
\begin{proof}[Proof of Theorem \ref{theorem:A_p_restrict_T_general}] 

Without loss of generality we can assume that $\{\mul_j\}_{j}$ is a finite family with cardinal $J\in \N$. We are going to prove the weak case. The strong case can be  obtained with minor modifications in the argument. 

Let $\{\psi_{m}\}_m$ be a family of  nonnegative functions in
$\mathcal{C}^\infty_c(\R^d)$, supported in the cube $[-1/2,1/2]^d$ such that
is an approximation of the  identity in $L^1(\T^d)$. We can also assume that
$\lim_{m} \psi_{m}*w_l(x)=w_l(x)$ a.e. $x\in [-1/2,1/2]^d$ for $l=0,1,2$ where $w_0=w$. 

Fixed $m\in\N$, Lemma \ref{lem:regularidad_peso_debil} yields
\begin{equation}\label{eq:technic_Maximal}
	\norm{M(f_1,f_2)}_{L^{p,\infty}(\R^d,\psi_m*w)}\leq c_p \mathfrak{N}\prod_{l=1,2}\norm{f_l}_{L^{p_l}(\R^d,\psi_m*w)}.
\end{equation}
Lemma \ref{lem:acotacion_peso} yields that for any $h\in \mathcal{C}_c^\infty(\R^d)$ such that $0\leq h\leq 1$ and $\int_{\R^d} h=1$, there exists an $n_m$ such that, for
any $n\geq n_m$, the conclusions of Lemma \ref{lem:approximaci�n}
hold for the periodic weight $\psi_m*w_l$ with $l=1,2$.

Consider now
\[
\mul_{j,n}(\xi)=   \widehat{K_{j,n}}(\xi_1,\xi_2)=(\widehat{\Phi_n}*{\mul_j})(\xi)\widehat{h_n}(\xi_1)\widehat{h_n}(\xi_2),\quad {\text{for $j,n\in \N$,}}
\]
where $\Phi_n=\varphi_n\otimes \varphi_n$ and $\varphi_n$ are functions as in Definition \ref{def:normalized}. Since $\varphi_n$ and $h_n$ are compactly supported it follows that $K_{j,n}\in \mathcal{C}_c^\infty(\R^{2d})$.
We also have that
\begin{equation}\label{eq:normalised_thm}
    \lim_n \widehat{K_{j,n}}(\xi_1,\xi_2)=\mul_j(\xi_1,\xi_2)\quad {\text{ for every $\xi_1,\xi_2\in \R^d$}},
\end{equation}
as $\mul_j$ is normalized and $\widehat{h_n}\to 1$. Furthermore  
$\norm{\mul_{j,n}}_{L^\infty(\R^{2d})}\leq \norm{\mul_{j}}_{L^\infty(\R^{2d})}$ , for any $j$ as $\norm{\widehat{\varphi_n}}_{L^1(\R^d)}\leq 1$ and $\Vert{\widehat{h_n}}\Vert_{L^\infty(\R^d)}\leq 1$. With these notations we have that
\begin{equation*}
    B_{K_{j,n}}(f_1,f_2)=T_{(\widehat{\varphi_n}\otimes \widehat{\varphi_n})*{\mul_j}}(h_n*f_1,h_n*f_2)\quad  \text{for any $f_1,f_2\in \mathcal{C}_c^\infty(\R^d)$}.
\end{equation*}
with $B_{K_{j,n}}$ defined as in \eqref{eq:BK}. 
Then, \eqref{eq:technic_Maximal}, Proposition \ref{prop:good_behaviour} and Lemma \ref{lem:approximaci�n} yield 
\[
    \begin{split}
    \norm{\sup_{1\leq j\leq J} \abs{B_{K_{j,n}}(f_1,f_2)}}_{L^{p,\infty}(\R^d,\psi_m*w)} &\leq
   2^{\frac 1 p}\mathfrak{c}_{\vec{p}}\mathfrak{N}  
    \prod_{l=1,2}\norm{f_l}_{L^{p_l}(\R^d,\psi_m*w_l)}\quad \text{ for every $n\geq n_m$.}%
    \end{split}
\]
Thus, Theorem \ref{theorem:A_p_restrict_T} yields that, for any $n\geq n_m$ and any $g_1,g_2\in P(\T^d)$
\[
    \begin{split}
    \norm{\sup_{1\leq j\leq J} \abs{\mathfrak{T}_{\mul_{j,n}\vert_{\Z^{2d}}}(g_1,g_2)}}_{L^{p,\infty}(\T^d,\psi_m*w)} &\leq
    2^{\frac 1 p} \mathfrak{c}_{\vec{p}}\mathfrak{N} 
    \prod_{l=1,2}\norm{g_l}_{L^{p_r}(\T^d,\psi_m*w_l)}. %
    \end{split}
\]
Since \eqref{eq:normalised_thm} implies that
\[
    \begin{split}
    \lim_n \mathfrak{T}_{\mul_{j,n}} (g_1,g_2)(\theta)=\lim_n \sum_{k\in \Z^d} \mul_{j,n}(k_1,k_2) \widehat{g}(k_1)\widehat{g}(k_2) e^{2\pi i (k_1+k_2) \theta}=\mathfrak{T}_{\mul_j} (g_1,g_2)(\theta),
    \end{split}
\]
Fatou's lemma yields
\[
    \begin{split}
    \norm{\sup_{1\leq j\leq J}\abs{\mathfrak{T}_{\mul_j} (g_1,g_2)}}_{L^{p,\infty}(\T^d, \psi_m*w)} \hspace{-12pt}&\leq
    \liminf_n \norm{\sup_{1\leq j\leq J}\abs{\mathfrak{T}_{\mul_{j,n}} (g_1,g_2)}}_{L^{p,\infty}(\T^d,
    \psi_m*w)}\\
    \leq
    & 2^{\frac 1 p}\mathfrak{c}_{\vec{p}} \mathfrak{N}\prod_{l=1,2}\norm{g_l}_{L^{p_l}(\T^d,
    \psi_m*w_l)}.
    \end{split}
\]
We can now let $m\to \infty$ in the previous inequality to obtain \eqref{eq:weak_periodic} by recourse to the fact that 
\[
     \norm{g_l}_{L^{p_l}(\T^d, \psi_m*w_l)}\leq \norm{g_l}_{L^\infty(\T^d)}
     \norm{\psi_m*w_l-w_l}_{L^1(\T^d)}+\norm{g_l}_{L^{p_l}(\T^d,w)}
\]
and that $\lim_m \norm{g_m*w_l-w_l}_{L^1(\T^d)}=0$, for $l=1,2$.
\end{proof}

\section{Consequences and applications}\label{section:consequences}

In this section we give some applications of Theorem \ref{theorem:A_p_restrict_T_general}. We start by recalling the definition of weights belonging to the $A_p(\R^d)$ class. We refer the reader to \cites{MR807149} for other properties and generalities of these weights.

\begin{defn}  We say that a weight $w$ belongs to the class $A_p(\R^d)$, and we
write $w\in A_p(\R^d)$ if,
\[
    \sup_Q \brkt{\frac{1}{\abs{Q}}\int_Q w(x)\; \dd x}\brkt{\frac{1}{\abs{Q}}\int_Q w(x)^{1/1-p}\;
    \dd x}^{p-1}<\infty,
\]
for $1<p<\infty$, and
\[
    \sup_Q \brkt{\frac{1}{\abs{Q}}\int_Q w(x)\; \dd x}\norm{
    w^{-1}\chi_Q}_{\infty}<+\infty,
\]
 where the supremum is taken over the family of cubes
$Q$ with sides parallel to the coordinate axis. 

We denote by $A_p(\T^d)$ the family of weights belonging to $A_p(\R^d)$ such that are $1$-periodic in each variable.
\end{defn}

\subsection{Multilinear Coifman-Meyer symbols }We can apply our results to multilinear multipliers that give rise to multilinear Calder\'on-Zygmund operators. 
More precisely, as an immediate corollary of  our Theorem \ref{theorem:A_p_restrict_T_general} we obtain the following periodic counterpart of  L. Grafakos and R. Torres result \cite[Corollary 3.2 and Remark 3.6]{MR1947875} for multipliers.
\begin{cor} Let $1<p_1, \ldots, p_N<\infty$, $1/p_1+\ldots+1/p_N=1/p$ and define $p_0=\min(p_1,\ldots, p_N)$. Let $w\in A_{p_0}(\T^d)$ and let $\mul\in \mathcal{C}^\infty(\R^{Nd}\setminus \{0\})\cap \mathcal{C}(\R^{Nd})$ satisfying 
\[
	\abs{\d^{\alpha_1}_{\xi_1}\ldots \d^{\alpha_N}_{\xi_N} \mul(\xi_1,\ldots, \xi_N)}\leq C_{\alpha_1,\ldots, \alpha_N} \brkt{\abs{\xi_1}+\ldots \abs{\xi_N}}^{-\brkt{\abs{\alpha_1}+\ldots +\abs{\alpha_N}}},
\]
for any multi-indices $\alpha_1,\ldots,\alpha_N$. Let $K(x)=\widehat{\mul}(-x)$ and, for each $j\geq 0$, let $\mul_j$ be the Fourier transform of the truncated kernel $K \chi_{\set{\abs{y}>2^{-j}}}$. Define $T_j$ to be the multiplier operator associated to $\mul_j\vert_{\Z^{Nd}}$. Consider 
\[
	T_* (g_1,\ldots, g_N)(x)=\sup_{j\geq 0} \abs{T_j (g_1,\ldots, g_N)(x)},
\]
Then we have 
\[
	T_*:L^{p_1}(\T^d,w)\times \ldots \times L^{p_N}(\T^d,w) \to L^{p}(\T^d,w),
\]
and the same holds for $\mathfrak{T}_{\mul\vert_{\Z^{Nd}}}$. Moreover, if $w\in A_1(\T^d)$ then 
\[
	\mathfrak{T}_{\mul\vert_{\Z^{Nd}}}:L^{1}(\T^d,w)\times \ldots \times L^{1}(\T^d,w) \to L^{1/m,\infty}(\T^d,w).
\]
\end{cor}

\subsection{H\"ormander-Mihlin type multilinear multipliers}

We start by recalling the definition of Sobolev-type spaces. To this end, let $\psi\in \mathcal{C}^\infty_c(\R^d)$ be such that 
\[
	\supp \psi \subset \set{\xi\in \R^{Nd}: 1/2\leq \abs{\xi}\leq 2}, \quad \sum_{k\in \Z} \psi (\xi 2^{-k})=1,\quad \text{$\forall \xi\in \R^{Nd} \setminus\{0\}$}.
\]
For $\mul\in L^\infty(\R^d)$ let 
\[
	\mul_k(\xi_1,\ldots, \xi_N)=\mul(2^k \xi_1,\ldots, 2^k \xi_N)\psi(\xi_1,\ldots,\xi_N), \quad k\in \Z, \quad \xi_1,\ldots,\xi_N\in \R^d.
\]
With this notation define, for $s,s_1,\ldots, s_N\geq 0$
\[
	\|\mul_{k}\|_{H^{s}(\mathbb{R}^{Nn})}=\Big(\int_{\mathbb{R}^{Nn}} \brkt{1+\abs{\xi}^2}^s  |{\mul}_{k}(\xi)|^2\dd \xi \Big)^{1/2},
\]
and 
\[
	\|\mul_{k}\|_{H^{(s_1,\dots,
s_m)}(\mathbb{R}^{Nn})}=\Big(\int_{\mathbb{R}^{Nn}}\prod_{j=1}^N \brkt{1+\abs{\xi_j}^2}^{s_j}|{\mul}_{k}(\xi_1,\dots,\xi_N)|^2 \dd \xi_1\ldots \dd \xi_N\Big)^{1/2}.
\]

We can apply our Theorem \ref{theorem:A_p_restrict_T_general} to transfer the results of M. Fujita and N. Tomita \cite[Theorem 1.2 and Theorem 6.2]{MR2958938} %
to the periodic case. 
\begin{cor} Let $1<p_1,\ldots,p_N<\infty$,  $1/p_1+\ldots+1/p_N=1/p$, $Nd/2<s\leq N d$ and  $r=\min\{p_1,\ldots, p_N\}$. Assume that either 
\begin{enumerate}
	\item $r>Nn/s$ and $w\in A_{sr/N}(\T^d)$, or
	\item $r<(Nn/s)^\prime$, $1<p<\infty$ and $w^{1-p'}\in A_{p's/Nd}(\T^d)$.
\end{enumerate}
If $\mul\in L^\infty(\R^{Nd})$ is normalized and satisfies  $\sup_{k\in \mathbb{Z}}\|\mul_k\|_{H^{s}(\mathbb{R}^{Nd})}<\infty$, then $ \mathfrak{T}_{\mul\vert_{\Z^{Nd}}}$ is bounded from $L^{p_1}(\T^n,w)\times\ldots \times L^{p_N}(\T^d,w)$ to $L^{p}(\T^d,w)$.
\end{cor}

\begin{cor} Let $1<p_1,\ldots,p_N<\infty$,  $1/p_1+\ldots+1/p_N=1/p$ and $n/2<s_j\leq n$ for $1\leq j\leq N$. Assume that 
\[
	p_j>d/s_j \quad \text{and}\quad w_j\in A_{p_j s_j/d}(\T^d)\quad \text{for}\quad 1\leq j\leq N.
\]
If $\mul\in L^\infty(\R^{Nd})$ is normalized and satisfies 
$
	\sup_{k\in \mathbb{Z}}\|\mul_k\|_{H^{(s_1,\ldots s_N)}(\mathbb{R}^{Nd})}<\infty,
$
then $ \mathfrak{T}_{\mul\vert_{\Z^{Nd}}}$  is bounded from $L^{p_1}(\T^d,w_1)\times\ldots \times L^{p_N}(\T^d,w_N)$ to $L^{p}(\T^d,w)$ with $w=\prod_{j=1}^N w_j^{p/p_j}$.
\end{cor}

\subsection{Multipliers in  H\"ormander multilinear class $S^{m}_{\rho,0}$} 

We can obtain the following periodic counterparts of N. Michalowski, D. Rule and W. Staubach \cite[Theorem 3.3]{MRS} for multipliers in a H\"ormander class $S^m_{\rho,0}(\R^{Nd})$ ( that is multipliers satisfying  \eqref{ks1} below).

\begin{thm} \label{multi_one}
Fix $p_j \in [1,2]$ for $j=1,\dots,N$ and let $\mul\in\mathcal{C}^\infty(\R^{Nd} )$ satisfying
\begin{equation} \label{ks1}
|\partial_{\xi_1}^{\alpha_1}\ldots \partial_{\xi_N}^{\alpha_N} \mul(\xi_1,\ldots,\xi_N)| \leq C_{\alpha_1,\ldots,\alpha_N} \left(1+\abs{\xi_1}+\ldots+\abs{\xi_N}\right)^{m-\rho (|
\alpha_1|+\ldots+\abs{\alpha_N})},
\end{equation}
for any multi-indices $\alpha_1,\ldots,\alpha_N$,  with $0 \leq \rho \leq 1$ and $m < (\rho - 1)\sum_{j=1}^N\frac{n}{p_j}$. 

Then  for $p_j < q_j <\infty$ and $r > 0$ such that $\frac{1}{r} = \sum_{j=1}^N \frac{1}{q_j}$, $\mathfrak{T}_{\mul\vert_{\Z^{Nd}}}$ is a bounded operator from $L^{q_1}\brkt{\T^d,w_1} \times \dots \times L^{q_N}\brkt{\T^d,w_N}$ to $L^{r}\brkt{\T^d, w}$ whenever $w_j \in A_{q_j/p_j}(\T^d)$ if $q_j < \infty$ for $j=1,\dots,N$, and $w = \prod_{j=1}^N w_j^{r/q_j}$.
\end{thm}

It is well know that any amplitude in the multilinear H\"ormander class $S^0_{1,0}$ determines a multilinear Calder\'on-Zygmund operator (see \cite{MR1880324}). Then we can transfer the results in \cite{MR2483720}*{Corollary 3.9} to the periodic case for obtaining the following result for multipliers in that class and weights satisfying \eqref{eq:ap_vec} below. These weights are said to belong to the so called $A_{\vec{p}}$ class (see \cite{MR2483720}*{Theorem 3.6}).

\begin{cor} Let $1\leq p_1,\ldots, p_N<\infty$,  $1/p_1+\ldots+1/p_2=1/p$ and $\mul\in \mathcal{C}^\infty(\R^{Nd})$ satisfying   \eqref{ks1} with $m=0$ and $\rho=1$.

Let $w_1,\ldots, w_N$ be $1$-periodic weights satisfying 
\begin{equation}\label{eq:ap_vec}
	w_j^{1-p_j'}\in A_{Np_j}(\T^d) \quad j=1,\ldots, N,\qquad v_{\vec{w}}\in A_{Np}(\T^d),
\end{equation}
where $v_{\vec{w}}=\prod_{j=1}^N w_{j}^{p/p_j}$ and, when $p_j=1$,
the condition $w_j^{1-p_j'}\in A_{Np_j}$ is understood as $w_j^{1/N}\in A_1(\T^d)$.

\begin{enumerate}
	\item If $1<p_j<\infty$, $j=1\ldots, N$, then 
	\[
		\norm{\mathfrak{T}_{\mul\vert_{\Z^{Nd}}}(g_1\ldots,g_N)}_{L^p(\T^d,v_{\vec{w}})} \leq C \prod_{j=1}^N \norm{g_j}_{L^{p_j}(\T^d,w_j)}.
	\]
	\item If $1\leq p_j<\infty$, $j=1\ldots, N$, and at least one of the $p_j=1$, then 
	\[
		\norm{\mathfrak{T}_{\mul\vert_{\Z^{Nd}}}(g_1\ldots,g_N)}_{L^{p,\infty}(\T^d,v_{\vec{w}})} \leq C \prod_{j=1}^N \norm{g_j}_{L^{p_j}(\T^d,w_j)}.
	\]			
\end{enumerate}
\end{cor}

\end{document}